\documentclass{amsart}

\usepackage{amssymb}
\usepackage{latexsym}
\usepackage{amsmath}
\usepackage{euscript}
\usepackage{graphics}
\usepackage{todonotes}

\def\dD{{\mathbb D}}

   \def\dT{{\mathbb T}}

\def\bm\chi{\mbox{\boldmath$\chi$}}

\let\xker=\ker \def\ker{{\xker\,}}

\def\sg{\operatorname{sign}}

\def\sg{\operatorname{sign}}

\newtheorem{theorem}{Theorem}[section]
\newtheorem{proposition}[theorem]{Proposition}
\newtheorem{corollary}[theorem]{Corollary}
\newtheorem{lemma}[theorem]{Lemma}
\newtheorem{definition}[theorem]{Definition}

\numberwithin{equation}{section}

\newcommand{\bbN}{\mathbb{N}}
\newcommand{\bbZ}{\mathbb{Z}}
\newcommand{\eitheta}{e^{i\theta}}

\newcommand{\nri}{n\rightarrow\infty}
\DeclareMathOperator{\Rl}{Re}

\date{\today}

\author[M. Derevyagin]{Maxim Derevyagin}
\address{
Maxim Derevyagin\\
University of Mississippi\\
Department of Mathematics\\
Hume Hall 305 \\ 
P. O. Box 1848 \\
University, MS 38677-1848, USA }
\email{derevyagin.m@gmail.com}

\author[B. Simanek]{Brian Simanek}

\address{Brian Simanek\\
Baylor University\\
Department of Mathematics\\
One Bear Place \#97328\\
Waco, TX 76798-7328, USA} 
\email{Brian\_Simanek@baylor.edu}


\begin{document}

\title[On Szeg\H{o}'s theorem for a nonclassical case]{On Szeg\H{o}'s theorem for a nonclassical case}

\begin{abstract}
In this paper we prove Szeg\H{o}'s Theorem for the case when a finite number of Verblunsky coefficients lie outside the closed unit disk.  Although a form of this result was already proved by A.L. Sakhnovich, we use a very different method, which shows that the OPUC machinery can still be applied to deal with such nonclassical cases.  The basic tool we use is Khrushchev's formula that in the classical case relates the absolutely continuous part of the measure and the $N$-th iterate of the Schur algorithm. It is noteworthy that Khrushchev's formula makes the proof short and extremely transparent. Also, we discuss Verblunsky's theorem for the case in question.
\end{abstract}

\maketitle

\section{Introduction}

Let $\{\alpha_n\}_{n=0}^{\infty}$ be a sequence of complex numbers such that
\begin{equation}\label{DefCond}
|\alpha_n|<1, \quad n=0, 1, 2, \dots.
\end{equation}
For any nonnegative integer $n$ one can define a monic polynomial $\Phi_{n+1}$ of degree $n+1$ by the following Szeg\H{o} recurrence:
\begin{equation}\label{SzRec}
\begin{split}
\Phi_{n+1}(z)=z\Phi_n(z)-\overline{\alpha}_n\Phi_n^*(z)\\
\Phi_{n+1}^*(z)=\Phi_n^*(z)-\alpha_nz\Phi_n(z),
\end{split}
\end{equation}
provided that we set the initial condition to be
\begin{equation}\label{InCond}
\Phi_0(z)=1
\end{equation}
and $\Phi_n^*$ is the polynomial reversed to $\Phi_n$, that is,
\begin{equation}\label{RevPol}
\Phi_n^*(z)=z^n\overline{\Phi_n(1/\overline{z})}.
\end{equation}
In fact, any sequence of complex numbers $\{\alpha_n\}_{n=0}^{\infty}$ generates a family of polynomials $\{\Phi_n\}_{n=0}^{\infty}$ via \eqref{SzRec}.  Verblunsky's Theorem  \cite[Theorem 1.7.11]{OPUC1} tells us that the condition \eqref{DefCond} guarantees the existence of a probability measure $\mu$ on the unit circle $\dT$ such that
\begin{equation} \label{Orth}
\int_{0}^{2\pi} e^{-ij\theta} 
\Phi_n (e^{i\theta}) \, d\mu(\theta) =0, \quad j=0,1,2,\dots, n-1.
\end{equation} 
In other words, $\Phi_n$ is the $n$-th monic orthogonal polynomial with respect to the positive measure $\mu$ supported 
on $\dT$. To elucidate the connection between probability measures and sequences $\{\alpha_n\}_{n=0}^{\infty}$ satisfying \eqref{DefCond}, we must discuss the Schur algorithm.

Recall that a Schur function $f$ is an analytic function mapping the unit disk $\dD$ to its closure $\overline{\dD}$, that is, 
\begin{equation*} 
\sup_{z\in\dD}\, |f(z)| \leq 1.
\end{equation*} 
Then \eqref{DefCond} is necessary and sufficient for the existence of a sequence of Schur functions $\{f_n\}_{n=0}^{\infty}$ that are related in the following way:
\begin{equation} \label{SchurAlg} 
f_n(z) = \frac{\alpha_n + zf_{n+1}(z)}{1+\bar\alpha_n zf_{n+1}(z)}, \quad n=0, 1, 2, \dots. 
\end{equation}
The recursion \eqref{SchurAlg} allows us to find the Schur function $f=f_0$ knowing the sequence $\{\alpha_n\}_{n=0}^{\infty}$ (for more details see \cite[Section 1.3.6]{OPUC1}). Once the Schur function $f$ is determined, it gives rise to the Carath\'eodory function $F$ by the formula
\begin{equation} \label{Carathe} 
F(z) = \frac{1+zf(z)}{1-zf(z)}.
\end{equation}
Recall that a Carath\'eodory function $F$ is an analytic function on $\dD$ which obeys
\begin{equation} \label{CarCond} 
F(0)=1, \qquad  \Rl F(z) >0\,\mbox{ when }\,z\in\mathbb{D}. 
\end{equation}
Indeed, for the function $F$ defined by \eqref{Carathe} it is easily seen that 
\begin{equation}\label{ReOfCar}
\Rl F(z)=\frac{1-|zf(z)|^2}{|1-zf(z)|^2}>0.
\end{equation} 
Finally, the measure $\mu$ can be recovered by using the fact that Carath\'eodory functions admit the representation \cite{A61}
\begin{equation} \label{IntCara} 
F(z) = \int_0^{2\pi} \frac{e^{i\theta}+z}{e^{i\theta}-z}\, d\mu(\theta)
\end{equation} 
for some non-trivial (that is, infinitely supported) probability measure $\mu$.

One of the central questions studied in the theory of orthogonal polynomials on the unit circle (hereafter abbreviated by OPUC) is the question of how properties of the measure $\mu$ correspond to properties of the coefficients $\{\alpha_n\}_{n=0}^{\infty}$, often called the Verblunsky coefficients.  In particular, we are especially interested in such a result, known as \textit{Szeg\H{o}'s Theorem}, which is among the most celebrated results in the theory of OPUC.
To formulate the theorem, let us introduce the decomposition
\begin{equation*}  
d\mu = w(\theta)\, \frac{d\theta}{2\pi} + d\mu_s, 
\end{equation*} 
where $w\in L^1 (\partial\dD, \frac{d\theta}{2\pi})$ and $d\mu_s$ is singular with respect to $d\theta/2\pi$. Then Szeg\H{o}'s theorem reads
\begin{equation} \label{SzegoThDef}
\prod_{j=0}^\infty (1-|\alpha_j|^2) = \exp \biggl( \int_0^{2\pi} \log (w(\theta)) \, \frac{d\theta}{2\pi}\biggr).
\end{equation} 
By \cite[Equation 1.3.32]{OPUC1}, we know that
\begin{equation*}%\label{CarToMeasure}
\Rl F(e^{i\theta})=\frac{1-|f(e^{i\theta})|^2}{|1-e^{i\theta}f(e^{i\theta})|^2}=w(\theta),
\end{equation*}
so formula \eqref{SzegoThDef} can be rewritten as 
\begin{equation} \label{SzegoThDefCar}
\prod_{j=0}^\infty (1-|\alpha_j|^2) = \exp \biggl( \int_0^{2\pi} \log 
(\Rl F(e^{i\theta})) \, \frac{d\theta}{2\pi}\biggr),
\end{equation} 
which can also be translated to Boyd's Theorem for the underlying Schur function
\begin{equation} \label{BoydTh}
\prod_{j=0}^\infty (1-|\alpha_j|^2) = \exp \biggl( \int_0^{2\pi} \log (1-|f(e^{i\theta})|^2) \, \frac{d\theta}{2\pi}\biggr).
\end{equation} 
Details about these results and much more can be found in \cite[Chapter 2]{OPUC1} (see also \cite[Chapter 8]{Kh08}).

One of the goals of this note is to show that one can use techniques from the theory of OPUC to prove Szeg\H{o}'s Theorem when the condition \eqref{DefCond} fails to hold for a finite number of Verblunsky coefficients. Namely, {\bf in what follows we only consider   sequences $\{\alpha_n\}_{n=0}^{\infty}$ of complex numbers for which there exists a natural number $N$ such that}
\begin{equation}\label{IndCond}
\begin{split}
|\alpha_n|&\ne 1, \quad n=0,1,2, \dots N-1,\\
|\alpha_n|&<1, \quad n=N, N+1, N+2, \dots.
\end{split}
\end{equation}
Polynomials generated by such sequences via (\ref{SzRec}) have been previously studied in \cite{Gero,Krein,MG,Zayed} (see also \cite{DD04,DD07} where an analogous theory in a slightly more general form was developed for the real line case).

One additional motivation for studying polynomials generated by sequences satisfying \eqref{IndCond} comes from the theory of orthogonal polynomials on the real line.  Verblunsky's Theorem bears an obvious resemblance to Favard's Theorem, which asserts that for every pair of real sequences $\{a_n\}_{n\in\bbN}$, $\{b_n\}_{n\in\bbN}$ with each $a_n>0$ there exists a corresponding non-trivial probability measure supported on the real line (see \cite[Theorem 2.5.2]{Ibook}).  The correspondence manifests itself via the recursion relation satisfied by the orthonormal polynomials for the corresponding measure.  Favard's Theorem was generalized by Shohat (heavily influenced by work of Boas) in \cite{Shohat} to show that given any pair of real sequences $\{a_n\}_{n\in\bbN}$ and $\{b_n\}_{n\in\bbN}$, where each $a_n\neq0$, there is a signed measure on the real line so that the sequence of monic polynomials generated by those sequences and the appropriate recursion relation is orthogonal with respect to that signed measure.  %This generalizes Favard's Theorem by allowing negative values of $a_n$.
The condition $a_n\neq0$ is equivalent to the invertibility of all $N\times N$ leading principal submatrices of the moment matrix of the corresponding measure.  In the setting of OPUC, it is the condition $|\alpha_n|\neq1$ that is similarly necessary and sufficient for invertibility of leading principal submatrices of the corresponding moment matrix (see \cite[Theorem 1.5.11]{OPUC1} or \cite{MS}\footnote{See also \cite{ADL07} and \cite{DGK86}, where it is shown how to handle the situation when the conditon $|\alpha_n|\neq1$ fails in the context of the Schur algorithm for moment and interpolation problems.}).  Therefore, this is the analogous condition we impose in \eqref{IndCond}.  However, we will show that a generalization of Verblunsky's Theorem analogous to Shohat's generalization of Favard's Theorem does not exist in the unit circle setting (but see \cite[Section 3]{Zayed} and also \cite{Gero,VZ}). Nevertheless, we give a result that could be considered an analog of Verblunsky's Theorem in our case.

\smallskip

In the next section, we will review some important aspects of the theory of OPUC and adapt several important formulas to the setting of sequences satisfying \eqref{IndCond} so that we may prove Szeg\H{o}'s Theorem in Section \ref{szt}.  In Section \ref{verb} we will demonstrate that sequences that satisfy \eqref{IndCond} do not in general correspond to signed measures on the unit circle.

\section{Khrushchev's formula and pseudo-Carath\'eodry functions}\label{kf}

It is quite clear that the condition \eqref{IndCond} does not affect the algebraic part of the theory. For instance, from a sequence that satisfies \eqref{IndCond} it is possible to construct a family of polynomials $\Phi_n$ by using \eqref{SzRec} and \eqref{InCond}. Also, mimicking the classical Schur algorithm we can introduce the sequence of functions
\begin{equation} \label{SchurAlgInd} 
f_n(z) = \frac{\alpha_n + zf_{n+1}(z)}{1+\bar\alpha_n zf_{n+1}(z)}, \quad n=0, 1, 2, \dots,
\end{equation}
which can be considered as the generalized Schur algorithm (cf. \cite{ADL07} and \cite{DGK86}, where the generalized Schur algorithm is defined differently but we keep the form \eqref{SchurAlgInd} in order to freely use the algebraic part of the theory of OPUC).  The functions $\{f_n\}_{n=0}^{\infty}$ are not Schur functions in general. Nevertheless, the second part of \eqref{IndCond} ensures us that $f_N$, $f_{N+1}$, $f_{N+2}$, \dots are Schur functions, which is a consequence of the theory developed by Schur (for instance, see \cite[Section 9]{DK03}). Therefore, combining the first $N$ iterates \eqref{SchurAlgInd} leads to the following representation of the function $f=f_0$ 
\begin{equation}\label{MainRepr}
f(z)=\frac{A_{N-1}(z)+zB_{N-1}^*(z)f_N(z)}{B_{N-1}(z)+zA_{N-1}^*(z)f_N(z)},
\end{equation}
where $A_{N-1}$, $B_{N-1}$ are polynomials, $A_{N-1}^*$, $B_{N-1}^*$ are the reversed polynomials defined by \eqref{RevPol}, and $f_N(z)$ is a Schur function. Hence, although the function $f$ is not a Schur function, it has the representation \eqref{MainRepr}, which gives us some insight about the function and it suggests the way to proceed with the theory.

Before going into details, let us recall a few more algebraic properties of the polynomials $A_{N-1}$ and $B_{N-1}$, which are called the Wall polynomials. The first property that we need is 
\begin{equation}\label{WallonT}
|B_{N-1}(z)|^2-|A_{N-1}(z)|^2=\prod_{j=0}^{N-1}(1-|\alpha_j|^2), \quad z\in\dT,
\end{equation}  
which is essentially a consequence of the fact that \eqref{MainRepr} was obtained as the composition of the iterates from \eqref{SchurAlgInd} (see \cite[Section 1.3.8]{OPUC1} or \cite[Section 8.1]{Kh08}).  For notational convenience, we define the sequence $\{\omega_n\}_{n\geq0}$ by
\[
\omega_{n}:=\prod_{j=0}^{n}(1-|\alpha_j|^2).
\]
It should be stressed here that $\omega_{n}$ could be positive or negative unlike in the classical case where it is always positive. 

Another important property is the following relation between the Wall polynomials and the sequence $\{\Phi_n\}_{n=0}^{\infty}$, which is sometime called the Pint\'{e}r-Nevai formula (see \cite{PN} or \cite[Section 8.1]{Kh08})
\begin{equation}\label{WallToOPUC}
\Phi_N(z)=zB_{N-1}^*(z)-A_{N-1}^*(z),\qquad \Phi_N^*(z)=B_{N-1}(z)-zA_{N-1}(z).
\end{equation}
As we will see later, the following adaptation of Khrushchev's formula, which was used for the theory developed in \cite{Kh01}, is the key to our analysis (for more information about the formula see \cite[Section 8.3]{Kh08} and \cite[Section 9.2]{OPUC2}).

\begin{proposition}[Khrushev's formula]\label{formula}
Let $\{\alpha_n\}_{n=0}^{\infty}$ be a sequence of coefficients that satisfies \eqref{IndCond} and let $\Phi_N$ be the $N$-th polynomial generated by this sequence via \eqref{SzRec}.  Let $f_N$ be the classical Schur function corresponding to the sequence $\{\alpha_n\}_{n=N}^{\infty}$ and let $f$ be defined by \eqref{MainRepr}. Then for the function
\begin{equation} \label{CaratheInd} 
F(z) := \frac{1+zf(z)}{1-zf(z)},
\end{equation}
which does not have to be a Carath\'eodory function, we have that 
\begin{equation}\label{Khrushev}
\Rl F(z)=\omega_{N-1}\frac{(1-|f_N(z)|^2)}{|\Phi_N^*(z)-z\Phi_N(z)f_N(z)|^2}
\end{equation}
for Lebesgue almost every $z\in\dT$.
\end{proposition}

\begin{proof}
The combination of formulas \eqref{MainRepr} and \eqref{CaratheInd} shows that $\Rl F$ exists almost everywhere on $\dT$.
Then, due to \eqref{CaratheInd} we obviously arrive at 
\begin{equation*}
\Rl F(z)=\frac{1-|zf(z)|^2}{|1-zf(z)|^2}
\end{equation*}
the same way we did in the classical theory. To get to \eqref{Khrushev} from the latter relation we need to use \eqref{WallonT}, \eqref{MainRepr}, and \eqref{WallToOPUC}.  The details are similar to the proof of \cite[Lemma 8.47]{Kh08} or \cite[Theorem 9.2.4]{OPUC2}.
\end{proof} 

One of the consequences of \eqref{Khrushev} is that the analogue $F$ of the Carath\'eodory function turns out to be a multiple of a pseudo-Carath\'eodory function.  Pseudo-Carath\'eodory functions were introduced by Delsarte, Genin, and Kamp in \cite{DGK86} who were motivated by a substantial interest in these functions in the applied mathematics literature at that time (see also \cite{Dym88} for a certain development of the matrix case).  Before we can relate $F$ defined by \eqref{CaratheInd} to the class of pseudo-Carath\'eodory function, we give the precise definition of this class.

\begin{definition}[\cite{DGK86}] It is said that $F$ is a pseudo-Carath\'eodory function if 
\begin{enumerate}
\item[(i)] $F$ is the ratio of two bounded and analytic functions in $\dD$
\item[(ii)] $\Rl F \ge 0$ almost everywhere on $\dT$.
\end{enumerate}
\end{definition}

Notice that in this paper we are mainly concerned with pseudo-Carath\'eodory functions for which either $F(0)=1$ or $F(0)=-1$. This condition is inherited from an analogous condition for the classical case, which is the first relation in \eqref{CarCond}, since we are adopting the classical technique here.

It also worth mentioning that a subclass of pseudo-Carath\'eodory functions was independently studied by Krein and Langer (for instance, see \cite{KL77}) under the name \textit{generalized Carath\'eodory functions}. The motivation for the work in \cite{KL77} was the spectral theory of self-adjoint operators in Pontryagin spaces.

Now we can state the result about the functions in question.

\begin{theorem}\label{eref}
Assume the hypotheses of Proposition \ref{formula} and let $F$ be the function defined by \eqref{CaratheInd}.  Let $\epsilon_{N-1}$ be the sign of $\omega_{N-1}$, that is, $\epsilon_{N-1}=\sg\omega_{N-1}$. Then the function $\epsilon_{N-1}F$ is a pseudo-Carath\'eodory function. 
\end{theorem}

\begin{proof}
After substituting \eqref{MainRepr} into \eqref{CaratheInd} one can see that $F$ is the ratio of two bounded analytic functions in $\dD$ since $f_N$ is a Schur function and $A_{N-1}$, $B_{N-1}$ are polynomials. It remains to observe that $\Rl (\epsilon_{N-1}F)=\epsilon_{N-1}\Rl F$ so \eqref{Khrushev} implies that $\Rl (\epsilon_{N-1}F) \ge 0$ almost everywhere on $\dT$.
\end{proof}

%\noindent\textit{Remark.} It is worth mentioning that $\epsilon_{N-1}F$ is actually a pseudo-Carath\'eodory function of finite index following the terminology of Delsarte, Genin, and Kamp \cite{DGK86} or it is a generalized Carath\'eodory function in the sense of Krein and Langer \cite{KL77}.

\section{Szeg\H{o}'s Theorem}\label{szt}

In this section we will prove Szeg\H{o}'s theorem for the nonclassical case of Verblunsky coefficients \eqref{IndCond}
that we consider in this paper. Actually, the result we are proving here was already obtained by Sakhnovich in \cite{S00} by using the method of operator identities.  The main interest in our results is the novelty of our approach to the problem.  We simplify the proof of Szeg\H{o}'s theorem in the nonstandard case and organize the theory in a more transparent way.  Specifically, the way we look at the case when only a finite number of Verblunsky coefficients are outside the closed unit disk makes Szeg\H{o}'s theorem almost obvious to those who are familiar with the Khrushchev analysis of OPUC. 

The version of Szeg\H{o}'s Theorem that we will prove is an adaptation of \eqref{SzegoThDefCar} to the case \eqref{IndCond}. The first obstacle to reformulating \eqref{SzegoThDefCar} in this more general setting is that it is not immediately clear if we can define $\log\Rl F$ on $\dT$.  We will overcome this by appealing to Theorem \ref{eref}, which tells us that $\epsilon_{N-1}\Rl F$ is positive on $\dT$, which makes $\Rl F$ of a constant sign on $\dT$ and hence $\log\Rl F$ is well-defined on $\dT$. Another delicate issue is that $\Rl F$ on $\dT$ is no longer the boundary value of a harmonic function in $\dD$ because $F$ could have poles in $\dD$, and hence $\Rl F$ does not have to be harmonic in $\dD$.  We will overcome this problem through a more precise understanding of the nature of the poles of $F$, to which we now turn our attention.

Let us again recall a standard fact from the algebraic theory of OPUC. Define the sequence $\{\Psi_n\}_{n=0}^{\infty}$ of second kind polynomials by
\begin{equation}\label{WallToPsi}
\Psi_N(z):=zB_{N-1}^*(z)+A_{N-1}^*(z),\qquad \Psi_N^*(z):=B_{N-1}(z)+zA_{N-1}(z).
\end{equation} 
These polynomials are generated by \eqref{SzRec} and \eqref{InCond} with the coefficients $\{\alpha_n\}_{n=0}^{\infty}$ replaced by $\{-\alpha_n\}_{n=0}^{\infty}$ (see \cite[Section 8.1]{Kh08} or \cite[Section 3.2]{OPUC1}).

\begin{proposition}\label{PropPoles}
Let $f$ be the function defined by \eqref{MainRepr}, where the Verblunsky coefficients $\{\alpha_n\}_{n=0}^{\infty}$  satisfy \eqref{IndCond}. Then the corresponding function $F$ defined by \eqref{CaratheInd} can be represented in the form
\begin{equation} \label{CaratheIndRepr} 
F(z) = \frac{\Psi_N^*(z)+z\Psi_N(z)f_N(z)}{\Phi_N^*(z)-z\Phi_N(z)f_N(z)},
\end{equation}
where $f_N$ is the $N$-th iterate of the generalized Schur algorithm \eqref{SchurAlgInd} and is a Schur function.
Moreover, the function $F$ is meromorphic in $\dD$ with a number of poles in $\dD$ not exceeding the number of zeros of $\Phi_N^*$ in $\dD$.  The poles of $F$ are exactly the zeros of $\Phi_N^*(z)-z\Phi_N(z)f_N(z)$ that belong to $\dD$. 
\end{proposition}

\begin{proof}
Formula \eqref{CaratheIndRepr} is an immediate consequence of \eqref{MainRepr}, \eqref{CaratheInd}, \eqref{WallToOPUC}, and \eqref{WallToPsi}. Then it is clear that the poles of $F$ are generated by the zeros of the bounded analytic function $R(z):=\Phi_N^*(z)-z\Phi_N(z)f_N(z)$.  Next, define $A_n^{(N)}$ and $B_n^{(N)}$ to be the $n$-th Wall polynomials corresponding to the sequence of Verblunsky coefficients $\{\alpha_n\}_{n=N}^{\infty}$.  By \cite[Section 1.3.8]{OPUC1}, we know that $A_n^{(N)}/B_n^{(N)}$ converges to $f_N$ uniformly on compact subsets of $\dD$ as $\nri$.  It also follows from \cite[Equation 1.3.84]{OPUC1} that $|A_n^{(N)}(z)/B_n^{(N)}(z)|<1$ when $|z|=1$ and hence Rouch\'{e}'s Theorem implies
\[
R_n(z):=\Phi_N^*(z)-z\Phi_N(z)A_n^{(N)}(z)/B_n^{(N)}(z)
\]
has the same number of zeros in $\dD$ as $\Phi_N^*$.  Sending $\nri$ and using the aforementioned uniform convergence shows that $R(z)$ has a number of zeros in $\dD$ not exceeding the number of zeros of $\Phi_N^*$ in the open unit disk.%\footnote{An alternative proof of this fact can be deduced from the index theory of pseudo-Carath\'eodry functions described in \cite{DGK86}.}.

The last thing to prove is that the poles of $F$ are exactly the zeros of $R$, which means that we have to show that $\Psi_N^*(z)+z\Psi_N(z)f_N(z)$ and $R(z)$ have no common zeros. Indeed, the fact that they do not have common zeros in $\dD\setminus\{0\}$ directly follows from the relation 
 \begin{equation}\label{det1}
 \det\begin{pmatrix}
 \Phi_N(z)&\Psi_N(z)\\
 \Phi_N^*(z)&-\Psi_N^*(z)
 \end{pmatrix}=2z^N\omega_{N-1}=2z^N\prod_{j=0}^{N-1}(1-|\alpha_j|^2),
 \end{equation}
which is a consequence of the Schur algorithm  and for instance can be found in \cite[Section 8.1]{Kh08} or \cite[Proposition 3.2.2]{OPUC1}. Finally, we notice that $R$ does not vanish at $0$ because $R(0)=\Phi_N^*(0)=1$.
\end{proof}

The final preliminary results that we will need concern the zeros of the polynomials $\{\Phi_n\}_{n=0}^{\infty}$.  We begin with the following elementary result.

\begin{proposition}\label{zeros}
Let $\{\alpha_n\}_{n=0}^{\infty}$ satisfy \eqref{IndCond} and let $\{\Phi_n(z)\}_{n\geq0}$ be the sequence of polynomials generated by the recursion \eqref{SzRec} and \eqref{InCond}.  If $|\alpha_n|<1$, then the polynomial $\Phi_{n+1}(z)$ has the same number of zeros in $\dD$ as $z\Phi_n(z)$.  If $|\alpha_n|>1$, then $\Phi_{n+1}(z)$ has the same number of zeros in $\dD$ as $\Phi_n^*(z)$.
\end{proposition}

\noindent\textit{Remark.}  Results related to Proposition \ref{zeros} can be found in \cite{Krein} and \cite[Remark 4.10]{MG}.

\begin{proof}
Since each $\alpha_n$ satisfies $|\alpha_n|\neq1$, we know that $|\Phi_n(z)|=|\Phi_n^*(z)|\neq0$ whenever $|z|=1$, so the desired conclusions follow from Rouch\'{e}'s Theorem and the first relation in \eqref{SzRec}.
\end{proof}

\begin{corollary}\label{both2}
Let $\{\Phi_n(z)\}_{n\geq0}$ be as in Proposition \ref{zeros} and suppose $|\alpha_k|>1$ for some $k<N$.  If $m\geq N$, then the number of zeros of $\Phi_m^*(z)$ in $\dD$ is strictly positive and independent of $m$.
\end{corollary}

We can say more about the zeros of $\Phi_n^*$ for large $n$.  The following lemma tells us that the zeros of $\Phi_n^*$ inside $\dD$ will either tend to the poles of $F$ in $\dD$ or to the boundary of the unit disk as $\nri$.

\begin{lemma}\label{pzeros}
Let $\{\alpha_n\}_{n=0}^{\infty}$ satisfy \eqref{IndCond}, let $\{\Phi_n\}_{n=0}^{\infty}$ be the corresponding sequence of monic polynomials, and let $F$ be the function defined as in \eqref{CaratheIndRepr}.  Denote by $\{\lambda_j\}_{j=0}^m$ the poles of $F$ in $\dD$.
\begin{enumerate}
\item[i)] If $U$ is an open set contained in some compact $K\subset\dD$ and containing $j$ poles of $F$ (counting multiplicity), then $U$ contains $j$ zeros of $\Phi_n^*$ (counting multiplicity) for all sufficiently large $n$.
\item[ii)] If $K\subset\dD$ is compact, then $K$ contains at most $m+1$ zeros of $\Phi_n^*$ for all sufficiently large $n$.
\end{enumerate}
\end{lemma}

\begin{proof}
From \eqref{CaratheInd} and \eqref{CaratheIndRepr}, we may write
\begin{equation}\label{dom}
F(z)=\frac{(\Psi_N^*(z)-\Psi_N(z))+(\Psi_N^*(z)+\Psi_N(z))F_N(z)}{(\Phi_N^*(z)+\Phi_N(z))+(\Phi_N^*(z)-\Phi_N(z))F_N(z)},
\end{equation}
where $F_N$ is the Carath\'{e}odory function corresponding to the sequence of Verblunsky coefficients $\{\alpha_n\}_{n=N}^{\infty}$.  We have already seen that the poles of $F$ are the zeros of the denominator of \eqref{dom}.  In a similar way, if $\{P_n\}_{n=0}^{\infty}$ and $\{Q_n\}_{n=0}^{\infty}$ are the sequences of monic orthogonal and second kind polynomials respectively that correspond to the sequence of Verblunsky coefficients $\{\alpha_n\}_{n=N}^{\infty}$, then we may use the calculations in \cite[Section 3.4]{OPUC1} to write
\begin{equation*}
\Phi_{N+n}^*(z)=\frac{1}{2}\left[(\Phi_N^*(z)+\Phi_N(z))P_n^*(z)+(\Phi_N^*(z)-\Phi_N(z))Q_n^*(z)\right].
\end{equation*}
As we observed in \eqref{det1}, the polynomials $P_n^*$ and $Q_n^*$ share no common zeros.  Therefore, the zeros of $\Phi_{N+n}^*(z)$ in $\dD$ are precisely the zeros of
\begin{equation}\label{dom2}
(\Phi_N^*(z)+\Phi_N(z))+(\Phi_N^*(z)-\Phi_N(z))\frac{Q_n^*(z)}{P_n^*(z)}
\end{equation}
in $\dD$.  By \cite[Theorem 3.2.4]{OPUC1}, the expression in \eqref{dom2} converges to the denominator in \eqref{dom} uniformly on compact subsets of $\dD$ as $\nri$, so the zeros of $\Phi_{N+n}^*$ inside $\dD$ either converge to the poles of $F$ or the unit circle as $\nri$ and in such a way that implies the desired conclusion.
\end{proof}

Having completed the necessary preliminaries, we are now ready to state and prove the main result of the paper. We will see that the proof can be understood as a consequence of Khrushchev's formula and Szeg\H{o}'s Theorem for positive measures.

\begin{theorem}[Szeg\H{o}'s Theorem]\label{ST}
Let $\{\alpha_n\}_{n=0}^{\infty}$ be a sequence of Verblunsky coefficients that satisfies \eqref{IndCond} and let $F$ be the function defined by \eqref{CaratheInd} through \eqref{MainRepr}. Let $\{\lambda_j\}_{j=0}^{m}$ be the collection of poles of $F$ in $\dD$ each listed as many times as its multiplicity. Then
\begin{equation}\label{SzegoThIndCar}
\prod_{j=0}^\infty (1-|\alpha_j|^2) = \prod_{j=0}^m|\lambda_j|^{-2}\exp \biggl( \int_0^{2\pi} \log 
(\Rl F(e^{i\theta})) \, \frac{d\theta}{2\pi}\biggr).
\end{equation}
\end{theorem}

\begin{proof}
We split the proof into two cases.  The first case is when the product on the left-hand side of \eqref{SzegoThIndCar} is non-zero.

Let us start by considering the function $F$. Formula \eqref{Khrushev} reads
\begin{equation*}
\Rl F(z)=\omega_{N-1}\frac{(1-|f_N(z)|^2)}{|\Phi_N^*(z)-z\Phi_N(z)f_N(z)|^2}, \qquad\qquad z\in\dT.
\end{equation*}
After taking the logarithm we arrive at
\begin{equation}\label{log}
\quad \log (\Rl F(z))= \log\omega_{N-1} + \log (1-|f_N(z)|^2) - \log |\Phi_N^*(z)-z\Phi_N(z)f_N(z)|^2. 
\end{equation}
We will proceed to evaluate the integral
\[
\int_0^{2\pi} \log (\Rl F(e^{i\theta})) \, \frac{d\theta}{2\pi},
\]
and then take its exponential.  By \eqref{log}, this integral splits into three pieces and the first one is simply a constant, which gives
\begin{equation}\label{Int1}
\exp \biggl( \int_0^{2\pi} \log \omega_{N-1}
 \, \frac{d\theta}{2\pi}\biggr)=\omega_{N-1}=\prod_{j=0}^{N-1}(1-|\alpha_j|^2)
\end{equation}
and we recall that this expression could be positive or negative.  To evaluate the second term we use Boyd's Theorem \eqref{BoydTh}
\begin{equation}\label{Int2}
\exp \biggl( \int_0^{2\pi} \log (1-|f_N(e^{i\theta})|^2) \, \frac{d\theta}{2\pi}\biggr)=
\prod_{j=N}^\infty (1-|\alpha_j|^2) 
\end{equation}

It remains to evaluate
\begin{equation}\label{lastpiece}
\int_0^{2\pi}\log(|\Phi_N^*(e^{i\theta})-e^{i\theta}\Phi_N(e^{i\theta})f_N(e^{i\theta})|^2)\frac{d\theta}{2\pi}.
\end{equation}
To do so, we notice that in \eqref{IndCond}, we do not assume that $N$ is the smallest natural number for which the second condition there holds.  Therefore, the relation \eqref{Khrushev} remains valid if we replace $N$ by $N+n$, where $n$ is any nonnegative integer.  Moreover, we can repeat all the above reasoning with $N$ replaced by $N+n$ . Clearly $\log(\Rl F(z))$ is independent of $n$, and hence so is its integral.  Furthermore, we notice that the product of \eqref{Int1} and \eqref{Int2} is independent of $n$, and therefore in \eqref{lastpiece} we may replace $N$ by $N+n$, where $n$ is a nonnegative integer and not change the value of the integral.  Therefore, it suffices to calculate
\begin{equation*}
\lim_{\nri}\int_0^{2\pi}\log(|\Phi_{N+n}^*(e^{i\theta})-e^{i\theta}\Phi_{N+n}(e^{i\theta})f_{N+n}(e^{i\theta})|^2)\frac{d\theta}{2\pi}.
\end{equation*}
To evaluate this integral, write $\Phi_{N+n}(z)=C_{N+n}(z)P_{N+n}(z)$, where $C_{N+n}$ is a monic polynomial that is non-vanishing in the closed unit disk and $P_{N+n}$ is a monic polynomial that is non-vanishing in the compliment of the open unit disk. Notice that since $C_{N+n}$ and $P_{N+n}$ are monic, it holds that $C_{N+n}^*(0)=1$ and $P_{N+n}^*(0)=1$. We can then write
\begin{align}
\nonumber&\log|\Phi_{N+n}^*(e^{i\theta})-e^{i\theta}\Phi_{N+n}(e^{i\theta})f_{N+n}(e^{i\theta})|^2=\log|P_{N+n}^*(\eitheta)|^2\\
%&\qquad=\log|C_{N+n}^*(e^{i\theta})P_{N+n}^*(\eitheta)-e^{i\theta}C_{N+n}(e^{i\theta})P_{N+n}(\eitheta)f_{N+n}(e^{i\theta})|^2\\
\label{3terms}&\qquad\qquad\qquad\qquad+\log|C_{N+n}^*(\eitheta)|^2+\log\left|1-\frac{\eitheta P_{N+n}(\eitheta)}{P_{N+n}^*(\eitheta)}\,\frac{C_{N+n}(\eitheta)}{C_{N+n}^*(\eitheta)}\,f_{N+n}(\eitheta)\right|^2
\end{align}
The Mean Value Theorem immediately gives
\[
\int_0^{2\pi}\log|P_{N+n}^*(\eitheta)|^2\frac{d\theta}{2\pi}=\log|P_{N+n}^*(0)|=0,
\]
To evaluate the integral involving $C_{N+n}^*$, we recall that $C_{N+n}^*(0)=1$, so we write (for some $k\in\bbN$)
\[
C_{N+n}^*(z)=\prod_{j=1}^{k}\left(1-\frac{z}{\lambda_{j,n}}\right),
\]
where $\lambda_{j,n}\in\dD$ and $k$ is independent of $n$ by Corollary \ref{both2}.  This gives
\[
\begin{split}
\lim_{\nri}\int_0^{2\pi}\log|C_{N+n}^*(\eitheta)|^2\frac{d\theta}{2\pi}&=
\lim_{\nri}\int_0^{2\pi}\log|C_{N+n}(\eitheta)|^2\frac{d\theta}{2\pi}=\lim_{\nri}\log|C_{N+n}(0)|^2\\
&=\lim_{\nri}\log\left(\prod_{j=1}^{k}|\lambda_{j,n}|^{-2}\right)=\log\left(\prod_{j=0}^m|\lambda_{j}|^{-2}\right),
\end{split}
\]
where we used Lemma \ref{pzeros}.  To integrate the third term in (\ref{3terms}), notice that
\[
1-|f_{N+n}(\eitheta)|\leq\left|1-\frac{\eitheta P_{N+n}(\eitheta)}{P_{N+n}^*(\eitheta)}\,\frac{C_{N+n}(\eitheta)}{C_{N+n}^*(\eitheta)}f_{N+n}(\eitheta)\right|\leq1+|f_{N+n}(\eitheta)|
\]
and the same inequalities hold when we take a logarithm.  By writing
\[
\log\frac{1+|f_{N+n}(\eitheta)|}{1-|f_{N+n}(\eitheta)|}=\log(1-|f_{N+n}(\eitheta)|^2)-2\log(1-|f_{N+n}(\eitheta)|)
\]
and invoking \cite[Theorem 8.56]{Kh08} and Boyd's Theorem, we conclude that
\[
\lim_{\nri}\int_0^{2\pi}\log(1-|f_{N+n}(\eitheta)|)\frac{d\theta}{2\pi}=0
\]
and the same theorem then shows
\[
\lim_{\nri}\int_0^{2\pi}\log(1+|f_{N+n}(\eitheta)|)\frac{d\theta}{2\pi}=0.
\]
Therefore,
\begin{equation*}
\lim_{\nri}\int_0^{2\pi}\log\left|1-\frac{\eitheta P_{N+n}(\eitheta)}{P_{N+n}^*(\eitheta)}\,\frac{C_{N+n}(\eitheta)}{C_{N+n}^*(\eitheta)}f_{N+n}(\eitheta)\right|\frac{d\theta}{2\pi}=0
\end{equation*}
and hence
\begin{equation}\label{lastpiece3}
\exp\left(\int_0^{2\pi}\log(|\Phi_N^*(e^{i\theta})-e^{i\theta}\Phi_N(e^{i\theta})f_N(e^{i\theta})|^2)\frac{d\theta}{2\pi}\right)=\prod_{j=0}^m|\lambda_{j}|^{-2}.
\end{equation}
Combining the relations \eqref{Int1}, \eqref{Int2}, and \eqref{lastpiece3} completes the proof if the left-hand side of \eqref{SzegoThIndCar} is non-zero.

If the left-hand side of \eqref{SzegoThIndCar} is zero, then the product of \eqref{Int1} and \eqref{Int2} is zero, \cite[Theorem 17.17]{Rudin} shows that the integral \eqref{lastpiece} is finite, and $F$ does not have a pole at $0$, so the right-hand side of \eqref{SzegoThIndCar} is also zero and the desired equality still holds.
\end{proof}

We would like to stress again that, unlike the proof in \cite{S00}, our approach does not employ any result from the theory of generalized Carath\'eodory functions (or pseudo-Carath\'eodory functions) and, in fact, is based on rather elementary facts. Nevertheless, we could also make use of the factorization result from \cite{DGK86} to directly compute the integral \eqref{lastpiece} without taking limits.

\section{Verblunsky's Theorem}\label{verb}

In this section, we will prove the following result, which shows that Shohat's extension of Favard's Theorem has no analog in the the theory of OPUC.

\begin{theorem}\label{counter}
Consider the sequence of Verblunsky coefficients 
\[
\alpha_0,\ldots,\alpha_{N-1},0,0,\ldots,
\]
where $|\alpha_j|>1$ for some $j\in\{0,1,2,\ldots,N-1\}$.  The sequence of polynomials generated from these coefficients by the recursion (\ref{SzRec}) and (\ref{InCond}) is not an orthogonal set with respect to any signed measure $\mu$ on the unit circle.
\end{theorem}

\begin{proof} Suppose for contradiction that there is a signed measure $\mu$ such that the sequence of polynomials $\{\Phi_m(z)\}_{m\geq0}$ generated by (\ref{SzRec}) using $\{\alpha_0,\ldots,\alpha_{N-1},0,0,\ldots\}$ is orthogonal with respect to $\mu$.  Let $\{c_j\}_{j\in\bbZ}$ be the sequence of moments of $\mu$
\[
c_j=\int e^{-ij\theta}d\mu(\theta).
\]
Rescaling the measure does not change orthogonality, so we will assume that the measure satisfies $c_0=1$ as its $0^{th}$ moment (the case $c_0=0$ requires a separate argument, which we provide later).  %Define the sequence $\{\gamma_j\}_{j\geq0}$ by
%\[
%\gamma_j=\begin{cases}
%\alpha_j & j=0,1,\ldots,n\\
%0 & j>n.
%\end{cases}
%\]

Let $\{\Psi_n(z)\}_{n=0}^{\infty}$ be the sequence of second kind polynomials for the sequence $\{\alpha_0,\ldots,\alpha_{N-1},0,0,\ldots\}$.  From \cite[Equation 3.11]{Zayed}, we see that for $j=1,2,\ldots,m$, the $j$-th Maclaurin coefficient of the rational function $\frac{\Psi_m^*(z)}{\Phi_m^*(z)}$ is twice the $j$-th moment of the measure $\mu$ (see also \cite[Theorem 6.1]{Gero}).

For each natural number $m$, define $g_m(z):=\frac{\Psi_m^*(z)}{\Phi_m^*(z)}$.  By assumption, the functions $\left\{g_m(z)\right\}_{m=N}^{\infty}$ are all the same.  Therefore, for every $j\in\bbN$, the $j$-th Maclaurin coefficient of $g_{N}(z)$ is equal to twice the $j$-th moment of the measure $\mu$.  However, since $|\alpha_k|>1$ for some $k\leq N-1$, the polynomial $\Phi_{N}^*(z)$ vanishes somewhere inside $\dD$ (by Corollary \ref{both2}) and the relation \eqref{det1} implies $\Psi_{N}^*(z)$ shares no common zeros with $\Phi_{N}^*(z)$, so the radius of convergence of the Maclaurin series for $g_{N}$ must be strictly smaller than $1$.  It follows that the moments of $\mu$ grow exponentially through some subsequence, which is impossible.  Therefore, no such $\mu$ can exist with $c_0\neq0$.

Now suppose $\{\Phi_m\}_{m\geq0}$ are orthogonal with respect to a signed measure $\mu$ satisfying $c_0=0$.  It easily follows from the orthogonality relations that all of the moments of $\mu$ are $0$.  The F. and M. Riesz Theorem (see \cite[Theorem 17.13]{Rudin}) combined with \cite[Theorem 5.15]{Rudin} implies that $\mu$ is the zero measure, which gives a contradiction.
\end{proof}

Thus, one comes to the conclusion that signed measures are not natural objects in the context of the study of polynomials generated by sequences that satisfy \eqref{IndCond}. At the same time, the previous sections have shown that a convenient tool for analysis of the theory lying behind the condition \eqref{IndCond} is the collection of pseudo-Carath\'eodory functions. Moreover, non-trivial probability measures and Carath\'eodory functions are in one-to-one correspondence and it is through the use of pseudo-Carath\'eodory functions that we may extend Verblunsky's Theorem to the case we have been considering.

To do so, we need to recall the Schur algorithm \eqref{SchurAlgInd} and its generalization to our setting.  If we are given a function $F$ that is meromorphic on $\dD$ and satisfies $F(0)=1$, then we can define $f=f_0$ by \eqref{CaratheInd}.   We can then inductively define a sequence of functions $\{f_n\}_{n=0}^{\infty}$ by inverting \eqref{SchurAlgInd} and setting
\begin{equation}\label{genschur}
f_{n+1}(z)=\frac{1}{z}\cdot\frac{f_n(z)-f_n(0)}{1-\overline{f_n(0)}f_n(z)},\qquad n=0,1,2,\ldots,
\end{equation}
as long as $0$ is not a pole of $f_n$ for any $n\in\bbN\cup\{0\}$.  We then get a sequence of Verblunsky coefficients by taking $\{f_n(0)\}_{n=0}^{\infty}$.  Notice that if $|f_n(0)|=1$, then $0$ is a pole of $f_{n+1}$.  In this way, provided that $0$ is not a pole of $f_n$ for any $n$, one can pass from properly normalized meromorphic functions to sequences of complex numbers, though the sequence obtained in this way may not satisfy \eqref{IndCond}.
%Note that for an arbitrary non-rational pseudo-Carath\'eodory function the condition that $f_n(0)$ is defined and satisfies $|f_n(0)|\ne 1$ for all $n$ does not hold automatically and, consequently, this procedure cannot always be applied in the form that we describe here.
However, every sequence satisfying \eqref{IndCond} determines a meromorphic function that generates the given sequence through this procedure, as our next theorem shows.  In fact, the function so determined is a pseudo-Carath\'eodory function.

\begin{theorem}\label{ThStrange}
Suppose the sequence $\{\alpha_n\}_{n=0}^{\infty}$ satisfies \eqref{IndCond}.  There exists a pseudo-Carath\'eodory function $F_*$ such that the function
\begin{equation}\label{f0def}
f_0(z):=\frac{1}{z}\,\frac{F_*(z)-F_*(0)}{F_*(z)+F_*(0)}
\end{equation}
determines, via \eqref{genschur}, a sequence $\{f_n(z)\}_{n=0}^{\infty}$ that satisfies $f_n(0)=\alpha_n$ for all $n\geq0$.  If $G_*$ is another pseudo-Carath\'eodory function with the same property, then $G_*$ is a scalar multiple of $F_*$.
\end{theorem}

\noindent\textit{Remark.}  If $0$ is a pole of $F_*$, then the function $f_0$ determined by the formula \eqref{f0def} is defined to be $-1/z$.  With this understanding, the formula \eqref{f0def} is well-defined for all pseudo-Carath\'eodory functions, though the resulting function $f_0$ may not be analytic at $0$.

\smallskip

\noindent\textit{Remark.}  A result similar to our Theorem \ref{ThStrange} but for a modified version of the Schur algorithm was obtained in \cite[Theorem 3.2]{DW05}.

\begin{proof}
Given $\{\alpha_n\}_{n=0}^{\infty}$ satisfying \eqref{IndCond}, define the function $F$ as in \eqref{CaratheInd} so that $F_*:=\varepsilon_{N-1}F$ is a pseudo-Carath\'eodory function by Theorem \ref{eref}.  Notice that $\varepsilon_{N-1}=F_*(0)$.  If $f_N$ denotes the Schur function corresponding to $\{\alpha_n\}_{n=N}^{\infty}$, then by construction, the function $f_0$ defined by \eqref{f0def} must be given by \eqref{MainRepr}.  Since this identity was the result of the recursion \eqref{SchurAlgInd}, it must be the case that the functions $\{f_n(z)\}_{n=0}^{\infty}$ defined inductively by $f_0$ and \eqref{genschur} satisfy $f_n(0)=\alpha_n$ for all $n\geq0$ as desired.

To establish uniqueness, suppose $G_*$ is another pseudo-Carath\'eodory function with the desired property and $g_0$ the corresponding function defined by \eqref{f0def} with $F_*$ replaced by $G_*$.  In particular, we know that the recursion \eqref{genschur} with seed function $g_0$ never produces an iterate with a pole at $0$, so $G_*(0)$ exists and is non-zero.  Furthermore, the iterate $g_N$ is such that $|g_{N+n}(0)|<1$ for all $n=0,1,2,\ldots$ so $g_N$ must be a Schur function by the Schur algorithm \cite[Section 9]{DK03}.  From the classical theory of OPUC, it must be the \textit{unique} Schur function corresponding to $\{\alpha_n\}_{n=N}^{\infty}$, and hence $g_N=f_N$, which means $g_0$ is given by the right-hand side of \eqref{MainRepr} as well.  It follows that $F_*(z)/F_*(0)=G_*(z)/G_*(0)$ as desired.
\end{proof}

The converse to Theorem \ref{ThStrange} is not true, i.e. it is not the case that any pseudo-Carath\'{e}odory function produces a sequences that satisfies \eqref{IndCond} by means of the above procedure.  Indeed, consider the function $F_*(z)=2+1/B(z)$, where $B$ is a Blaschke product with infinitely many zeroes in $\dD$.  From \cite[Theorem 15.24]{Rudin} we know that $|B|=1$ almost everywhere on $\dT$. Hence, $F_*$ is a pseudo-Carath\'eodory function, but it has infinitely many poles in $\dD$ while we have shown that sequences that satisfy \eqref{IndCond} correspond to pseudo-Carath\'eodory functions with only finitely many poles in $\dD$ (see Proposition \ref{PropPoles}).  Thus, $F_*$ either produces a function $f_0$ via \eqref{f0def} such that the recursion \eqref{genschur} terminates after finitely many steps (because some $f_n$ has a pole at $0$), or the corresponding sequence $\{f_n(0)\}_{n=0}^{\infty}$ has infinitely many elements outside $\overline{\dD}$.  Therefore, we see that sequences satisfying \eqref{IndCond} are in one to one correspondence with a \textit{proper} subclass of appropriately normalized pseudo-Carath\'eodory functions, and it is to functions in this subclass that we can apply Theorem \ref{ST}.

The reader familiar with the theory of generalized Schur functions and augmented Schur parameters (as described in \cite{DW05}) may recall the generalized Schur algorithm that appears in \cite{AADL01} and is expounded in \cite{ADL07} (see also \cite{DGK86}).  The generalized Schur algorithm leads to the bijection between generalized Schur functions and augmented Schur parameters that is established in \cite[Theorem 3.2]{DW05}. It is in many ways superior to \eqref{genschur} in its scope and utility because it applies to sequences that do not satisfy \eqref{IndCond}.  However, the standard connection to the theory of OPUC no longer holds for the generalized Schur algorithm from \cite{ADL07} and we have seen that this connection is essential for our use of the classical Schur algorithm \eqref{SchurAlgInd} and \eqref{genschur} when proving Theorem \ref{ST}. A version of Szeg\H{o}'s Theorem that applies to all pseudo-Carath\'eodory functions remains elusive.

\bigskip

\noindent\textbf{Acknowledgements.}  The second author would like to thank Mourad Ismail for encouraging us to pursue the results that lead to Theorem \ref{counter}.

%\newpage

\end{document}